\documentclass[oneside,english,draft]{amsart}
\usepackage[T1]{fontenc}
\usepackage[latin9]{inputenc}
\usepackage{amstext}
\usepackage{amsthm}
\usepackage{amssymb}
\usepackage{esint}
\usepackage{setspace}

\makeatletter
\numberwithin{equation}{section}
\numberwithin{figure}{section}
\theoremstyle{plain}
\newtheorem{thm}{\protect\theoremname}
  \theoremstyle{plain}
  \newtheorem{prop}[thm]{\protect\propositionname}
  \theoremstyle{plain}
  \newtheorem{lem}[thm]{\protect\lemmaname}

\usepackage{babel}
\usepackage{comment}

\makeatother

\usepackage{babel}
  \providecommand{\lemmaname}{Lemma}
  \providecommand{\propositionname}{Proposition}
\providecommand{\theoremname}{Theorem}

\begin{document}

\title[Global stability of an endemic equilibrium]{Global stability of an SIS epidemic model with a finite infectious
period}

\author{Yukihiko Nakata, Gergely R\"ost}

\address{Y. Nakata\\Department of Mathematics, Shimane University, 1060 Nishikawatsu-cho, Matsue, Shimane, 690-8504, Japan}
\address{G. R\"ost\\Bolyai Institute, University of Szeged, H-6720 Szeged, Aradi v\'ertan\'uk tere 1., Hungary}
\email[Y. Nakata]{ynakata@riko.shimane-u.ac.jp}

\begin{abstract}
Assuming a general distribution for the sojourn time in the infectious
class, we consider an SIS type epidemic model formulated as a scalar
integral equation. We prove that the endemic equilibrium of the model
is globally asymptotically stable whenever it exists, solving the
conjecture of Hethcote and van den Driessche (1995) for the case of
nonfatal diseases. 
\end{abstract}


\maketitle

\section{Introduction}

In the paper \cite{Hethcote:1995}, Hethcote and van den Driessche
formulate an SIS type epidemic model and analyze its qualitative properties.
The model has a disease related death rate, which affects the size
of the total population. The sojourn time of the infectious states
is assumed to follow a general probability density function, though
the authors elaborate the case of a constant infectious period. In
Theorem 5.1 in the paper \cite{Hethcote:1995}, the authors show that
the endemic equilibrium is asymptotically stable when the disease
related death rate is zero and subsequently mention global stability
of the endemic equilibrium as an open problem. See also \cite{Hethcote:2000}
for a study of a similar SIS epidemic model, where the assumed population
demography is different from \cite{Hethcote:1995}.

The problem of global stability does not seem to be fully solved up
to now. In the recent paper \cite{Iggidr:2010} the authors analyze
monotonicity of the semiflow induced by the model proposed in \cite{Hethcote:1995}.
It is shown that the solution semiflow is monotone with respect to
the initial function when a key parameter, namely the basic
reproduction number $R_{0}$, is greater than two. Then the author
of \cite{Liu:2015} proves that many solutions, but not all solutions,
converge to the stable endemic equilibrium if $R_{0}>2$. In those
papers, to derive those results, a constant infectious period is assumed
in the model. Note that in \cite{Iggidr:2010,Liu:2015} the long time
behavior of the solutions for $1<R_{0}\leq2$ is not studied even for the case of constant delay.

Our aim of this paper is to prove that the endemic equilibrium is
indeed globally asymptotically stable for the epidemic model proposed
by Hethcote and van den Driessche in \cite{Hethcote:1995,Hethcote:2000}
for the case of nonfatal diseases. We prove that all positive solutions
tend to the endemic equilibrium if $R_{0}>1$, thus the endemic equilibrium
is \textit{always} globally asymptotically stable. Our proof allows
any probability function which the sojourn time of the infectious
state follows.

This paper is organized as follows. In Section 2 we introduce an SIS
epidemic model, corresponding to the model considered in \cite{Hethcote:1995,Hethcote:2000}
for the case of nonfatal diseases. The model is formulated as a scalar
integral equation describing time evolution of the infectious population.
In Section 3 we prove global stability of the disease free equilibrium.
In Section 4 we show that the endemic equilibrium is asymptotically
stable analyzing the corresponding characteristic equation. Here an
elementary proof for persistence of the solution is also given. We then
prove global attractivity of the endemic equilibrium. In Section 5
we discuss our results.

\section{Endemic model}

Following \cite{Hethcote:1995}, we introduce the SIS epidemic model
of which we analyze its global dynamics. Denote by $\mathcal{F}(a)$
the probability that infective individuals are infective up to infection-age
$a$ (time elapsed since infection). From the biological interpretation,
\[
\mathcal{F}:\left[0,h\right]\to\left[0,1\right]
\]
is nonincreasing with 
\[
\mathcal{F}(0)=1,\ \mathcal{F}(h)=0,
\]
where $h<\infty$ is the maximum infectious period. Without loss of
generality we assume that $\mathcal{F}(a)>0$ for $0<a<h$.

Let $S(t)$ and $I(t)$ be the proportions of the susceptible population and the
infective population at time $t$, respectively, so that 
\begin{equation}
1=S(t)+I(t)\label{eq:total}
\end{equation}
holds for any $t$. The density of infective individuals at time $t$
with respect to infection-age $a$ is given as 
\begin{equation}
i(t,a)=\beta S(t-a)I(t-a)\mathcal{F}(a)e^{-\mu a},\label{eq:individual}
\end{equation}
where $\mu$ is the natural mortality rate and $\beta$ is the constant
transmission coefficient. Note that the density of newly infective
individuals appearing in the host population at time $t$ is 
\[
i(t,0)=\beta S(t)I(t).
\]
Integrating both sides of (\ref{eq:individual}) we obtain the
following renewal equation 
\begin{equation}
I(t)=\beta\int_{0}^{h}S(t-a)I(t-a)\mathcal{F}(a)e^{-\mu a}da,\label{eq:I_RE}
\end{equation}
see also the integral equation (2.2) in \cite{Hethcote:1995}.

We scale the time so that the maximum infectious period becomes $1$
and then subsequently we recall that \eqref{eq:total} holds for $t\geq-1$
(so $I(t)\leq1$ for $t\geq-1$). Then we get a scalar renewal equation
for the proportion of infective population: 
\begin{equation}
I(t)=\beta\int_{0}^{1}\left(1-I(t-a)\right)I(t-a)\mathcal{F}(a)e^{-\mu a}da\label{eq:I_RE2}
\end{equation}
from (\ref{eq:I_RE}).

Since $\mathcal{F}$ is a monotone function, $\mathcal{F}$ is a bounded
variation function such that 
\[
-\int_{0}^{1}d\mathcal{F}(a)=\mathcal{F}(0)-\mathcal{F}(1)=1
\]
holds. The proportion of individuals who obtain susceptibility per
unit time at time $t$ is given as 
\[
-\beta\int_{0}^{1}S(t-a)I(t-a)e^{-\mu a}d\mathcal{F}(a),
\]
where the integral is Riemann-Stieljes integral. Differentiating the
equation (\ref{eq:I_RE2}) one also obtains the following delay differential
equation 
\begin{equation}
\frac{d}{dt}I(t)=\beta(1-I(t))I(t)+\beta\int_{0}^{1}(1-I(t-a))I(t-a)e^{-\mu a}d\mathcal{F}(a)-\mu I(t).\label{eq:I_DE}
\end{equation}
Equations \eqref{eq:I_RE2} and \eqref{eq:I_DE} with 
\begin{equation}
\mathcal{F}(a)=\begin{cases}
1, & 0\leq a<1\\
0, & a=1
\end{cases}\label{eq:step}
\end{equation}
can be found in \cite{Hethcote:1995,Hethcote:2000} for the case of
no disease induced death rate. 

Denote by $C([-1,0],\mathbb{R})$ the Banach space of continuous functions
mapping the interval $[-1,0]$ into $\mathbb{R}$ equipped with the
sup-norm. The initial function is chosen from the subset of continuous
functions: 
\[
Y:=\left\{ \phi\in C\left(\left[-1,0\right],\left[0,1\right]\right)\left|\phi(0)=G(\phi)\right.\right\} ,
\]
where $G:C\left(\left[-1,0\right],\left[0,1\right]\right)\to\mathbb{R}$
is defined as 
\[
G(\phi):=\beta\int_{0}^{1}\left(1-\phi(-a)\right)\phi(-a)\mathcal{F}(a)e^{-\mu a}da.
\]
The initial function for (\ref{eq:I_RE2}), and for (\ref{eq:I_DE}),
is given as 
\begin{equation}
I(\theta)=\psi(\theta),\ \theta\in\left[-1,0\right]\label{eq:IC}
\end{equation}
with $\psi\in Y$. We introduce a standard notation $I_{t}:[-1,0]\to\mathbb{R}_{+}$
defined via the relation $I_{t}(\theta)=I(t+\theta)$ for $\theta\in[-1,0]$.
It is straightforward to see that the set $Y$ is forward invariant
under the semiflow induced by (\ref{eq:I_RE2}) with the initial condition
(\ref{eq:IC}) i.e., $I_{t}\in Y,\ t>0.$

\section{Global stability of the disease free equilibrium}

The basic reproduction number is easily computed as 
\[
R_{0}=\beta\int_{0}^{1}\mathcal{F}(a)e^{-\mu a}da.
\]
We here prove global stability of the disease free equilibrium when
$R_{0}\leq1$ holds, by using similar ideas as in the proof of Theorem
4.2 in \cite{Hethcote:1995} for the case that $\mathcal{F}$ is given
by \eqref{eq:step}. 
\begin{thm}
\label{thm:gas1}Let us assume that $R_{0}\leq1$ holds. Then the
disease free equilibrium is globally asymptotically stable in $Y$.\end{thm}
\begin{proof}
Consider a sequence $\left\{ M_{n}\right\} _{n=0}^{\infty}$ determined by 
\[
M_{n}=R_{0}M_{n-1}\left(1-M_{n-1}\right)
\]
with $M_{0}=\frac{1}{4}$. It holds that $M_{n}$ monotonically
decreases and that $\lim_{n\to\infty}M_{n}=0$ if $R_{0}\leq1$ holds.
From the equation (\ref{eq:I_RE2}) we have 
\[
I(t)\leq\frac{1}{4}R_{0}\leq\frac{1}{4},\ t\in\left[0,1\right].
\]
One can see that 
\[
I(t)\leq M_{n},\ t\in[n,n+1],
\]
thus we get $\lim_{t\to\infty}I(t)=0$. Stability of the disease free
equilibrium follows from the estimation above. 
\end{proof}

\section{Global stability of the endemic equilibrium}

An endemic equilibrium emerges if $R_{0}>1$ holds. In Theorem 5.1
in \cite{Hethcote:1995} the endemic equilibrium is shown to be asymptotically
stable when $\mathcal{F}$ is given by \eqref{eq:step}, see also
Theorem 4 in \cite{Hethcote:2000}. In the following proposition we
show that the endemic equilibrium is asymptotically stable for any
$\mathcal{F}$. 

Let us write $\hat{p}$ to denote the constant function in $Y$ satisfying
$\hat{p}(\theta)=p$ for $\theta\in[-1,0]$. We then define 
\[
Y_{+}:=Y\setminus\left\{ \hat{0}\right\} .
\]
In the following, to generate a positive solution, we assume $\psi\in Y_{+}$. 
\begin{prop}
\label{prop:LAS}Let us assume that $R_{0}>1$ holds. There exists
a unique endemic equilibrium $\hat{I^{*}}\in Y_{+}$, where 
\begin{equation}
I^{*}=1-\frac{1}{R_{0}}.\label{eq:EE}
\end{equation}
The endemic equilibrium is asymptotically stable. \end{prop}
\begin{proof}
Let $R_{0}>1$ hold. The endemic equilibrium satisfies the following
equation: 
\[
1=\left(1-I\right)\beta\int_{0}^{1}\mathcal{F}(a)e^{-\mu a}da=R_{0}\left(1-I\right),
\]
so we get (\ref{eq:EE}). 

We apply the principle of linearized stability established in \cite{Diekmann:2007}
for Volterra functional equations. The Fr\'echet derivative of $G:C\left(\left[-1,0\right],\left[0,1\right]\right)\to\mathbb{R}$
evaluated at the endemic equilibrium $\hat{I^{*}}\in Y_{+}$ can be
computed as 
\[
DG(\hat{I^{*}})\phi=\left(1-2I^{*}\right)\beta\int_{0}^{1}\phi(-a)\mathcal{F}(a)e^{-\mu a}da.
\]
Thus the characteristic equation is 
\begin{equation}
1=\left(1-2I^{*}\right)\beta\int_{0}^{1}e^{-\lambda a}\mathcal{F}(a)e^{-\mu a}da,\ \lambda\in\mathbb{C}.\label{eq:cheq}
\end{equation}
Denote $\tilde{\mathcal{F}}(a)=\mathcal{F}(a)e^{-\mu a}$. Using the
partial integration 
\[
\int_{0}^{1}e^{-\lambda a}\tilde{\mathcal{F}}(a)da=\frac{1}{\lambda}\left(1+\int_{0}^{1}e^{-\lambda a}d\tilde{\mathcal{F}}(a)\right),
\]
one can deduce a priori bounds for the roots of the characteristic
equation. We now show that the real part of all roots of (\ref{eq:cheq})
is negative. First it can be easily seen that the characteristic equation
(\ref{eq:cheq}) does not have a root with positive real part for
small $I^{*}>0$. Suppose that there exists a root $\lambda=i\omega$
with $\omega>0$. We get the following two equations\begin{subequations}\label{eq:cheqREIM}
\begin{align}
1 & =\left(1-2I^{*}\right)\beta\int_{0}^{1}\tilde{\mathcal{F}}(a)\cos\left(\omega a\right)da.\label{eq:cheq1}\\
0 & =\int_{0}^{1}\tilde{\mathcal{F}}(a)\sin\left(\omega a\right)da.\label{eq:cheq2}
\end{align}
\end{subequations}Integrating by parts we have 
\begin{align*}
0 & =\int_{0}^{1}\tilde{\mathcal{F}}(a)\sin\left(\omega a\right)da.\\
 & =\frac{1}{\omega}\left(\left[-\tilde{\mathcal{F}}(a)\cos\left(\omega a\right)\right]_{0}^{1}+\int_{0}^{1}\cos\left(\omega a\right)d\tilde{\mathcal{F}}(a)\right)\\
 & =\frac{1}{\omega}\left(1+\int_{0}^{1}\cos\left(\omega a\right)d\tilde{\mathcal{F}}(a)\right),
\end{align*}
thus $\int_{0}^{1}\cos\left(\omega a\right)d\tilde{\mathcal{F}}(a)=-1$
follows. Noting that 
\[
\left(\int_{0}^{1}\cos\left(\omega a\right)d\tilde{\mathcal{F}}(a)\right)^{2}+\left(\int_{0}^{1}\sin\left(\omega a\right)d\tilde{\mathcal{F}}(a)\right)^{2}\leq1
\]
holds, we see $\int_{0}^{1}\sin\left(\omega a\right)d\tilde{\mathcal{F}}(a)=0$
holds. Since one has that 
\begin{align*}
\int_{0}^{1}\tilde{\mathcal{F}}(a)\cos\left(\omega a\right)da & =\frac{1}{\omega}\left(\left[\tilde{\mathcal{F}}(a)\sin\left(\omega a\right)\right]_{0}^{1}-\int_{0}^{1}\sin\left(\omega a\right)d\tilde{\mathcal{F}}(a)\right)\\
 & =-\frac{1}{\omega}\int_{0}^{1}\sin\left(\omega a\right)d\tilde{\mathcal{F}}(a),
\end{align*}
we finally obtain $\int_{0}^{1}\tilde{\mathcal{F}}(a)\cos\left(\omega a\right)da=0$.
Therefore we get a contradiction to the equality in (\ref{eq:cheq1}).
Consequently in the complex plane the characteristic roots do not
cross the imaginary axis. 
\end{proof}
Next we show persistence of the solution when $R_{0}>1$. 
\begin{prop}
\label{prop:persistence}Let $R_{0}>1$ holds. Then 
\[
\liminf_{t\to\infty}I(t)>0.
\]
\end{prop}
\begin{proof}
We can choose a $q\in(0,1)$ such that $qR_{0}>1$ holds. Assume that
there is a solution with $\psi(0)>0$ such that $\lim_{t\to\infty}I(t)=0$.
Then there is a $T$ such that for $t>T$, $I(t)<\frac{1}{2}$ and
$1-I(t)>q$ hold. There is a sequence $t_{n}\to\infty$ $(n\to\infty)$
such that 
\[
I(t_{n})=\min\{I(t):t\in[t_{n}-1,t_{n}]\}\hbox{ and }\lim_{n\to\infty}I(t_{n})=0.
\]
Then, for $t_{n}>T+1$ we have 
\[
I(t_{n})=\beta\int_{0}^{1}I(t_{n}-a)(1-I(t_{n}-a))\mathcal{F}(a)e^{-\mu a}da,
\]
and from the monotonicity of the quadratic map $x\mapsto x(1-x)$ on $[0,\frac{1}{2}]$,
we obtain 
\[
I(t_{n})\geq\beta\int_{0}^{1}I(t_{n})(1-I(t_{n}))\mathcal{F}(a)e^{-\mu a}da=R_{0}I(t_{n})(1-I(t_{n})).
\]
Since we have $1-I(t_{n})>q$, dividing by $I(t_{n})$ we find the
contradiction $1>qR_{0}.$ 
\end{proof}
We now introduce the notations 
\[
\underline{I}=\liminf_{t\to\infty}I(t),\quad \overline{I}=\limsup_{t\to\infty}I(t).
\]
From Proposition \ref{prop:persistence} one can easily see that 
\[
0<\underline{I}\leq\overline{I}\leq1.
\]
The proof for the global attractivity of the endemic equilibrium is
divided into two cases: $1<R_{0}\leq2$ and $2<R_{0}$. Using the
monotonicity of the quadratic map $x\mapsto x(1-x)$ on the two intervals
$\left[0,\frac{1}{2}\right]$ and $\left[\frac{1}{2},1\right]$ for
the two cases respectively, we show that $\underline{I}=\overline{I}=I^{*}$
holds.

First let us assume that $1<R_{0}\le2$ holds. We define the subspace
\[
Y_{1}:=\left\{ \phi\in Y\left|0\leq\phi(\theta)\leq\frac{1}{2},\ \theta\in\left[-1,0\right]\right.\right\} 
\]
and first prove that $Y_{1}$ attracts all positive solutions. 
\begin{lem}
\label{lem:inv1}Let us assume that $1<R_{0}\leq2$. Then 
\[
I(t)\leq\frac{1}{2},\ t\geq0.
\]
thus 
\[
I_{t}\in Y_{1},\ t>1
\]
for any $\psi\in Y_{+}$.\end{lem}
\begin{proof}
Since $(1-x)x\le\frac{1}{4}$ for $x\in\left[0,1\right]$ we can estimate
the equation (\ref{eq:I_RE2}) as 
\[
I(t)\leq\frac{1}{4}R_{0}\le\frac{1}{2}.
\]
Then solution $I_{t},\ t>1$ satisfies $0\leq I_{t}(\theta)=I(t+\theta)\leq\frac{1}{2}$
for $\theta\in\left[-1,0\right]$. 
\end{proof}
It is easy to see that 
\[
I^{*}=1-\frac{1}{R_{0}}\leq\frac{1}{2},
\]
when $1<R_{0}\le2$ holds. To show global attractivity we use the
monotonicity of the quadratic function: 
\[
0\leq x\leq y\le\frac{1}{2}\to(1-x)x\le(1-y)y.
\]

\begin{prop}
\label{prop:ga1}Let us assume that $1<R_{0}\leq2$. Then 
\[
\lim_{t\to\infty}I(t)=I^{*}
\]
for any $\psi\in Y_{+}$.\end{prop}
\begin{proof}
From Proposition \ref{prop:persistence} and Lemma \ref{lem:inv1}
it holds that 
\[
0<\underline{I}\le\overline{I}\leq\frac{1}{2}.
\]
It then holds that 
\[
\underline{I}\ge G(\hat {\underline{I}})=R_{0}\underline{I}\left(1-\underline{I}\right)
\]
while 
\[
\overline{I}\leq G(\hat {\overline{I}})=R_{0}\overline{I}\left(1-\overline{I}\right).
\]
Therefore we get the following inequality 
\[
R_{0}\left(1-\underline{I}\right)\leq1\leq R_{0}\left(1-\overline{I}\right),
\]
which implies 
\[
\overline{I}=\underline{I}=I^{*}=1-\frac{1}{R_{0}}.
\]

\end{proof}
Next we assume that $R_{0}>2$ holds. It will be shown that all positive
solutions are attracted to the set 
\[
Y_{2}=\left\{ \phi\in Y\left|\frac{1}{2}\leq\phi(\theta)\leq1,\ \theta\in\left[-1,0\right]\right.\right\} .
\]
In Iggidr et al. \cite{Iggidr:2010} it is shown that $Y_{2}$ is
an invariant set and that the semiflow is monotone in $Y_{2}$ when
$\mathcal{F}$ is the step function \eqref{eq:step}. 
\begin{lem}
\label{lem:inv2}Let us assume that $R_{0}>2$ holds. For any $\psi\in Y_{+}$
there exists $\tau$ such that 
\[
I_{t}\in Y_{2},\ t>\tau.
\]
\end{lem}
\begin{proof}
First let us show that there exists no solution that is bounded by
$\frac{1}{2}$ from above. To derive a contradiction, suppose that
there exists a solution such that 
\[
I(t)\leq\frac{1}{2}\ \text{for any }t\geq0.
\]
Then 
\begin{equation}
0<\underline{I}\leq\overline{I}\leq\frac{1}{2}\label{eq:sup_inf}
\end{equation}
holds. From (\ref{eq:I_RE2}) one sees that 
\[
\underline{I}\geq G(\hat{\underline{I}})=R_{0}\underline{I}(1-\underline{I}),
\]
which implies 
\[
\underline{I}\ge I^{*}=1-\frac{1}{R_{0}}>\frac{1}{2}.
\]
Thus we obtain a contradiction to (\ref{eq:sup_inf}), so the solution
can not stay in $Y_{1}$ and leaves the set $Y_{1}$ eventually. Now
consider a solution that oscillates about $\frac{1}{2}$. We show
that the solution never returns to the set $Y_{1}$ once it crosses
$\frac{1}{2}$. To see this we compute the differentiation of $I(t)$
when $I(t)=\frac{1}{2}$ holds. From (\ref{eq:I_DE}) we get 
\begin{align*}
I^{\prime}(t) & =\beta\left(1-I(t)\right)I(t)+\beta\int_{0}^{1}\left(1-I(t-a)\right)I(t-a)e^{-\mu a}d\mathcal{F}(a)-\mu I(t)\\
 & =\frac{1}{4}\beta+\beta\int_{0}^{1}\left(1-I(t-a)\right)I(t-a)e^{-\mu a}d\mathcal{F}(a)-\frac{1}{2}\mu.
\end{align*}
Since for any nonnegative solutions it holds that
\[
\int_{0}^{1}\left(1-I(t-a)\right)I(t-a)e^{-\mu a}d\mathcal{F}(a)\ge\frac{1}{4}\int_{0}^{1}e^{-\mu a}d\mathcal{F}(a),
\]
it follows 
\begin{align*}
I^{\prime}(t) & \geq\frac{1}{4}\beta+\frac{1}{4}\beta\int_{0}^{1}e^{-\mu a}d\mathcal{F}(a)-\frac{1}{2}\mu\\
 & =\frac{1}{2}\mu\left\{ \frac{1}{2}\frac{\beta}{\mu}\left(1+\int_{0}^{1}e^{-\mu a}d\mathcal{F}(a)\right)-1\right\} .
\end{align*}
Since 
\[
\frac{\beta}{\mu}\left(1+\int_{0}^{1}e^{-\mu a}d\mathcal{F}(a)\right)=\beta\int_{0}^{1}e^{-\mu a}\mathcal{F}(a)da=R_{0}
\]
holds, we obtain 
\[
I^{\prime}(t)\geq\frac{1}{2}\mu\left(\frac{1}{2}R_{0}-1\right)>0,
\]
which implies that if the solution once crosses $\frac{1}{2}$ then
it is bounded below from $\frac{1}{2}$. 
\end{proof}

\begin{prop}
\label{prop:ga2}Let us assume that $R_{0}>2$ holds. Then 
\[
\lim_{t\to\infty}I(t)=I^{*}
\]
for any $\psi\in Y_{+}$.\end{prop}
\begin{proof}
From Lemma \ref{lem:inv2} it is sufficient to consider a solution
in $Y_{2}$. Consider a sequence $\left\{ p_{n}\right\} _{n=1}^{\infty}$
such that $I(p_{n})\to\underline{I}$ as $n\to\infty$. We apply the
fluctuation lemma to the equation (\ref{eq:I_DE}) to get

\begin{align*}
0 & =\beta(1-\underline{I})\underline{I}+\beta\lim_{n\to\infty}\int_{0}^{1}\left(1-I(p_{n}-a)\right)I(p_{n}-a)e^{-\mu a}d\mathcal{F}(a)-\mu\underline{I}\\
 & \ge\beta(1-\underline{I})\underline{I}+\beta(1-\underline{I})\underline{I}\int_{0}^{1}e^{-\mu a}d\mathcal{F}(a)-\mu\underline{I}\\
 & =\mu\underline{I}\left\{ \frac{\beta}{\mu}\left(1+\int_{0}^{1}e^{-\mu a}d\mathcal{F}(a)\right)\left(1-\underline{I}\right)-1\right\} .
\end{align*}
From Proposition \ref{prop:persistence} we have $\underline{I}>0$.
Thus 
\[
0\ge R_{0}\left(1-\underline{I}\right)-1,
\]
which implies that $\underline{I}\ge I^{*}$. Similarly we can apply
the fluctuation lemma for the sequence of a solution tending to $\overline{I}$.
Then we obtain $\overline{I}\leq I{}^{*}$. Consequently 
\[
\underline{I}\ge I^{*}\geq\overline{I}
\]
holds and we obtain the conclusion. 
\end{proof}
Finally we obtain the following result, combining the results of global
attractivity of the endemic equilibrium established in Propositions
\ref{prop:ga1} and \ref{prop:ga2} with local stability shown in
Proposition \ref{prop:LAS}.
\begin{thm}
\label{thm:gas}Let us assume that $R_{0}>1$ holds. The endemic equilibrium
is globally asymptotically stable in $Y_{+}$.
\end{thm}

\section{Discussion}

In the paper \cite{Hethcote:1995}, Hethcote and van den Driessche
propose an SIS type epidemic model with disease induced death rate.
The model developed in \cite{Hethcote:1995} allows a general description
for the distribution of the infectious period among individuals. To
proceed the mathematical analysis it is then specified as a step function
\eqref{eq:step}, implying that for every infective individual the
infectious period is exactly same. The authors in \cite{Hethcote:1995}
study stability of the disease free equilibrium and the endemic equilibrium.
In Theorem 5.1 in \cite{Hethcote:1995} it is shown that the endemic
equilibrium is asymptotically stable if there is no disease induced
death rate. We solve the open problem mentioned in the paper: the
endemic equilibrium is \textit{globally} asymptotically stable if
there is no disease induced death rate. Moreover the global stability
holds for the model with any distribution of the infectious period.
The equation (\ref{eq:I_RE2}) is also derived from the model considered
in \cite{Hethcote:2000} for the case of nonfatal diseases. In the
first model considered in \cite{Hethcote:2000} the authors assume
constant recruitment for the total population, differently from the
model in \cite{Hethcote:1995}. 

In the paper \cite{Gripenberg:1981} Gripenberg studied a similar
integral equation describing disease transmission dynamics motivated
by the study in \cite{Diekmann:1982}. Conditions for global attractivity
of the equilibria are obtained. It involves technical conditions concerning
kernels denoting variable infectivity and distribution for the sojourn
period. The model considered in \cite{Diekmann:1982} is also an SIS
type epidemic model, where individual's infectivity varies according
to the progression of the infection-age (the time elapsed since the
infection). The authors in \cite{Diekmann:1982} deduce a characteristic
equation for an endemic equilibrium and show that destabilization
of the endemic equilibrium is possible. In those papers \cite{Diekmann:1982,Gripenberg:1981}
the integral equation is formulated in terms of newly infectives per
unit time at time $t$.

Here we assume that the maximum infectious period is finite. The assumption
facilitates our analysis as we have the equation with \textit{finite}
delay \cite{Diekmann:2007}. The global stability results would be
extended to the case that the maximum infectious period is infinite
i.e., $h=\infty$, using the fading memory space \cite{Diekmann:2012,Hale:1978}, see also \cite{gerSIMA,gerMBE} for the treatment of such cases in epidemiological models.

\subsection*{Acknowledgment}
We thank the reviewer for comments which improve the original manuscript.

The research was initiated during the second author's visiting to
the University of Tokyo on August, 2015 and completed upon the first author's visiting to Bolyai Institute, University of Szeged on February, 2016.
The first author's research visiting was supported by JSPS Bilateral Joint Research Project (Open Partnership).
The first author was supported by JSPS Fellows, No.268448 of Japan Society for the Promotion of Science and JSPS Grant-in-Aid for Young Scientists (B) 16K20976. The second author was supported by European Research Council StG Nr. 259559 and OTKA K109782.


\begin{thebibliography}{10}
\bibitem{Diekmann:2007}O. Diekmann, Ph. Getto, M. Gyllenberg, { \it Stability
and bifurcation analysis of Volterra functional equations in the light
of suns and stars}, SIAM J. Math. Anal. 39 (2007/08), no. 4, 1023--1069.

\bibitem{Diekmann:2012}O. Diekmann, M. Gyllenberg, {\it Equations with
infinite delay: blending the abstract and the concrete}, J. Diff. Eq.
252 (2012), no. 2, 819--851.

\bibitem{Diekmann:1982}O. Diekmann, R. Montijn, {\it Prelude to Hopf bifurcation
in an epidemic model: analysis of a characteristic equation associated
with a nonlinear Volterra integral equation}, J. Math. Biol. 14 (1982)
117--127.

\bibitem{gerSIMA} S. Gourley, G. R\"ost, H.~R. Thieme
Uniform persistence in a model for bluetongue dynamics
SIAM Journal on Mathematical Analysis  46 (2014) no. 2, 1160--1184 

\bibitem{Gripenberg:1981}G. Gripenberg, {\it On some epidemic models}, Quart. Appl. Math. 39 (1981/82), no. 3, 317--327.

\bibitem{Hale:1978}J.~K. Hale, J. Kato, {\it Phase space for retarded equations
with infinite delay}, Funkcial. Ekvac 21.1 (1978) 11--41.

\bibitem{Hethcote:1995}H.~W. Hethcote, P. van den Driessche, {\it An SIS
	epidemic model with variable population size and a delay}, J. Math.
Biol., 34 (1995) 177--194

\bibitem{Hethcote:2000}H.~W. Hethcote, P. van den Driessche, {\it Two SIS
epidemiologic models with delays}, J. Math. Bio. 40 (2000) 3--26.


\bibitem{Iggidr:2010}A. Iggidr, K. Niri, E. Ould Moulay Ely, {\it Fluctuations
in a SIS epidemic model with variable size population}, Appl. Math.
Comput., 217 (2010) 55--64

\bibitem{Liu:2015}B. Liu, {\it Convergence of an SIS epidemic model with
a constant delay}, Appl. Math. Lett. 49 (2015) 113--118

\bibitem{gerMBE} G. R\"ost, J. Wu,
SEIR epidemiological model with varying infectivity and infinite delay,
Math. Biosci. Eng., 5 (2008) no. 2, 389--402


\end{thebibliography}
\end{document}